\documentclass[12pt]{article}
\usepackage{amsmath, amsfonts, amsthm, amssymb, color, hyperref, extarrows, enumitem, nicefrac,blindtext, mathtools, cite, anyfontsize}

\newtheorem{theorem}{Theorem}[section]

\newtheorem{cor}[theorem]{Corollary}
\newtheorem{lem}[theorem]{Lemma}
\newtheorem{pro}[theorem]{Proposition}

\numberwithin{equation}{section}

\newtheorem{example}{Example}

\usepackage[hyperpageref]{backref}

\usepackage{cite}

\hypersetup{
	colorlinks   = true,
	citecolor    = magenta}
	
\begin{document}
\title{\vspace{-1cm} \bf Non-quadratic solutions to the   Monge-Amp\`ere equation  \rm}
\author{Yifei Pan \ \ and\ \   Yuan Zhang }
\date{}

\maketitle

\begin{abstract}
We construct ample  smooth strictly plurisubharmonic non-quadratic solutions to the   Monge-Amp\`ere equation  on either cylindrical type domains or the whole complex Euclidean space $\mathbb C^2$. Among these, the entire solutions defined on $\mathbb C^2$ induce flat  K\"ahler metrics,  as expected by a question of Calabi. In contrast,   those on cylindrical domains produce  a family of nowhere flat K\"ahler metrics. Beyond these smooth solutions, we also classify  solutions that are radially symmetric in one variable, which exhibit various types of  singularities. Finally,  we explore analogous solutions to Donaldson’s equation motivated by a result of He.

\end{abstract}

\renewcommand{\thefootnote}{\fnsymbol{footnote}}
\footnotetext{\hspace*{-7mm}
\begin{tabular}{@{}r@{}p{16.5cm}@{}}
& 2020 Mathematics Subject Classification. Primary  32W20; Secondary 32Q15. \\
& Key words and phrases. Monge-Amp\`ere equation, flat metric, radial, explicit examples, Donaldson’s equation.
 
\end{tabular}}

  \section{Introduction}

According to  fundamental works   by J\"orgens \cite{Jor}  in dimension 2 and by Calabi
\cite{Ca} and Pogorelov \cite{Po} in higher dimensions, any entire convex viscosity solution of the real Monge-Amp\`ere equation 
$$ \det(D^2 u) =1 \ \ \text{on}\ \ \mathbb R^m $$
 must be a quadratic function. See also,  for instance, \cite{Le}. The similar property  no longer holds  for entire plurisubharmonic solutions
to the complex Monge-Amp\`ere equation in $\mathbb C^n$:
\begin{equation}\label{mae}
    \det (\partial\bar\partial u) =1.
\end{equation}
Here $\partial\bar\partial u$ is the complex Hessian of $u$. Such  solutions give rise to   K\"ahler metrics via the complex Hessian      whose associated   volume forms are constant.
Calabi  posted in \cite{Ca2} the  question of whether these K\"ahler metrics are flat, which still remains open.

Motivated by those results and Calabi's question, we investigate   solutions to \eqref{mae} in $\mathbb C^2$ that depend quadratically  on    one  variable. Specifically, for    variables $(z, w)\in \mathbb C^2$, we focus on real-valued solutions that are quadratic only in the $z$ direction: 
\begin{equation}\label{smae}
    u(z, w) = a(w) |z|^2 + b(w) z^2 +\overline{b(w)z^2}+c(w) z+\overline{c(w)z}  +d(w), 
\end{equation}
where $a, b, c$ and $d$ are smooth functions of $w$, and $a>0$. This particular  form enables us to convert the fully nonlinear equation \eqref{mae}  to a system of   simpler semi-linear elliptic equations, whose solutions can be effectively approached. See \eqref{ma0}-\eqref{ma1}.   

In Section 2, by applying  existence theorems from \cite{PZ} for nonlinear systems of Poisson  equations, we   construct ample smooth, strictly plurisubharmonic  solutions over cylindrical-type domains of the form $\mathbb C\times D_R$ for any $ R>0$,   where  $D_R$ denotes the disc of radius $R$ in $\mathbb C$. See Theorem \ref{mat}. Each of these solutions possesses the    special  quadric  form given in \eqref{smae}, yet they are  not quadratic in the $w$-variable. Note that all these solutions can be trivially extended  to the $n$ dimensional case by adding quadratic terms $|z_3|^2+\cdots |z_n|^2$.

To fully understand Calabi's question in our setting, we begin in Section 3 by establishing an obstruction that prevents the  induced K\"ahler metric from being flat.

\begin{theorem}\label{obp}
Let $D$ be a domain in $\mathbb C$. Suppose $u$ is a plurisubharmonic   solution to  \eqref{mae} on $\mathbb C\times D$ of the form \eqref{smae}. Then the corresponding K\"ahler metric is flat on $\mathbb C\times D$ if and only if  $b$ is holomorphic on $D$.
\end{theorem}
Making use of this theorem  one can  construct solutions to \eqref{mae} whose associated K\"ahler metric is nowhere flat over cylindrical-type domains.  See Theorem \ref{cf} and Example \ref{ex}.  In particular, these examples demonstrate that Calabi's question does not hold on any bounded domains. 

In Section 4, we produce in Theorem \ref{mate} entire solutions of \eqref{mae}  of the special form \eqref{smae} on $\mathbb C^2$. Note that the K\"ahler metrics induced by all these solutions are flat, and therefore do not provide counterexamples to Calabi’s question. 
Although it remains unclear whether the question holds in full generality, we show in Theorem \ref{cal} that for  every entire solutions of the form \eqref{smae}, $\frac{\partial b}{\partial \bar w}$ must have zeros somewhere.   

In addition to the aforementioned smooth solutions,   we also explore   in Section 5 solutions that are radially symmetric  on $w$ and exhibit singularities. In fact, we derive    explicit expressions for all such  solutions. At the end of the section, we present a variety of examples  with distinct singular  behavior (see Examples \ref{ex1}-\ref{ex2}). Notably, these  recover several existing known examples in the literature, such as B{\l}ocki \cite{Bl}, He\cite{He1}, and Wang-Wang \cite{WW}. 
   
   In Section 6,   more general forms of solutions that are non-quadratic in both variables are discussed, for instance, by  replacing $z$ in \eqref{smae} with a holomorphic function $\phi$ of $z$. As shown in Theorem \ref{gen}, in order for this to yield a solution to \eqref{mae}, the function  $\phi$ must take a very rigid dichotomous form  as described in \eqref{5} after normalization. Solutions involving such $\phi$ are then constructed in Theorem \ref{gens}.    
   
While  investigating the geometric structure for the space of volume forms on compact Riemannian manifolds,  Donaldson   introduced the operator $u_{tt}\Delta u -|\nabla u_t|^2 $, 
where $(t, x) \in \mathbb R\times \mathbb R^m$, and the Laplacian $\Delta$ and  the gradient $\nabla$ are both with respect to the space $x$ variable. When $m=2$, by complexifying the $t$ direction, 
Donaldson's operator can be reduced to a special case of the complex  Monge-Amp\`ere operator  $\det (\partial\bar\partial u).  $ In Section 7,  we generalize a result of He \cite{He} on solutions to 
\begin{equation}\label{de}
     u_{tt}\Delta u -|\nabla u_t|^2 =1,
\end{equation}
and obtain a larger class of solutions on cylindrical domains $\mathbb R\times B_R$ that are quadratic  in the $t$ variable in Theorem \ref{thd}. Here $B_R$ is the ball of radius $R $ in $\mathbb R^m$. However, for entire solutions,  Theorem \ref{thde} shows that every entire solution of the form \eqref{qde} must  reduce to He's original case.

  \section{Solvability   on cylindrical domains}
We first derive conditions for the coefficients $a, b, c$ and $d  $ in \eqref{smae} so that they yield  solutions to \eqref{mae}. Since $u$ is real-valued, so are $a$ and $d$. 
A straightforward computation gives the complex Hessian of $u$ below.
\begin{equation*}
    \partial \bar\partial u:  =  \begin{bmatrix}
u_{\bar z z} &  u_{z\bar w}\\
u_{\bar z w}  &    u_{\bar w w}
\end{bmatrix}  =  \begin{bmatrix}
 a & \ \ a_{\bar w} \bar z + 2b_{\bar w}z + c_{\bar w} \\
a_{   w} z +2\bar b_{ w}\bar z + \bar c_{ w} & \ \ a_{\bar w w}|z|^2 + b_{\bar w w}z^2 +\overline{b_{\bar w w} z^2}+c_{\bar w w} z+\overline{c_{\bar w w}z }  +d_{\bar w w}
\end{bmatrix}. 
\end{equation*}
Here   $f_w: = \frac{\partial f}{\partial w}$ for a smooth function $f$, and similarly   $f_{\bar w} : = \frac{\partial f}{\partial \bar w}$.  Consequently, taking the determinant of  the Hessian, and  sorting it out according to the powers of $z$ and $\bar z$, we obtain 
\begin{equation*}
\begin{split}
        \det( \partial \bar\partial u) = &|z|^2\left(a a_{\bar w w} -|a_w|^2 - 4|b_{\bar w}|^2 \right) + z^2\left(ab_{\bar w w}- 2a_wb_{\bar w}\right) +\bar z^2\left( a\overline{b_{\bar w w}}- 2a_{\bar w} \bar b_{ w}  \right)\\
        & + z\left(ac_{\bar w w} - a_{   w} c_{\bar w} - 2 b_{\bar w}\bar c_{ w} \right) +\bar z\left(a\bar c_{\bar w w} - a_{ \bar   w} \bar c_{w} - 2 \bar b_{ w}c_{\bar  w}   \right) + \left( ad_{\bar w w}-|c_{\bar w}|^2\right).
        \end{split}
\end{equation*}
Thus any solution $u$ to \eqref{mae} of the form \eqref{smae} should satisfy the following system of semilinear differential equations:
\begin{equation}\label{ma0}
   \left\{  \begin{aligned}
       &a a_{\bar w w} = |a_w|^2 + 4|b_{\bar w}|^2;\\
       &ab_{\bar w w} = 2a_wb_{\bar w};\\
       & ac_{\bar w w} =  a_{   w} c_{\bar w} + 2 b_{\bar w}\bar c_{ w}; \\
       & ad_{\bar w w} =|c_{\bar w}|^2+1.
    \end{aligned}\right.
\end{equation}

 To further simplify  \eqref{ma0}, noting that  $a>0$, let  
 \begin{equation}\label{tra}
     \tilde a:  = \ln a,
 \end{equation} 
 or equivalently, $ a =e^{\tilde a} $. Then $a_w =  e^{\tilde a} \tilde a_w$ and $a_{\bar w w} = e^{\tilde a} \left(\tilde a_{\bar w w} +   |\tilde a_{  w}|^2\right)$. This transformation allows us to   rewrite the system to be  
  \begin{equation}\label{ma1}
   \left\{  \begin{aligned}
      &\tilde a_{\bar w w} =   4e^{-2\tilde a} |b_{\bar w}|^2;\\
       &b_{\bar w w} = 2\tilde a_wb_{\bar w};\\
       &c_{\bar w w} =  \tilde a_{   w} c_{\bar w} + 2 e^{-\tilde a} b_{\bar w}\bar c_{ w}; \\
       &d_{\bar w w} =e^{-\tilde a}\left(|c_{\bar w}|^2+1\right).
    \end{aligned}\right.
\end{equation}
Namely, any solution to \eqref{ma1} on a domain $D\subset \mathbb C$ leads to a solution $u$ to \eqref{mae} on $\mathbb C\times D$. 

 To show the local existence of solutions to \eqref{ma1}, one can make use of  the following existence theorem for general  nonlinear systems of Poisson equations in \cite{PZ}. 
 
 \begin{theorem}\cite[Theorem 1.4]{PZ}\label{pz1}
Let  $F =(F_1,\ldots, F_N) $ be a $C^{1, \alpha}$ vector-valued function in $ \mathbb R^m\times \mathbb R^{N}\times \mathbb R^{mN}$ for some $0<\alpha<1$. Given any initial jets $(c_0, c_1)\in \mathbb R^N\times \mathbb R^{mN}$, there exist  infinitely many $C^{2, \alpha}$  solutions $v = (v_1, \ldots, v_N)$   satisfying
\begin{equation*}
    \left\{\begin{aligned}
        &\Delta v = F(\cdot, v, \nabla v);\\
        & v(0)=c_0;\\
        &\nabla v(0) = c_1 
    \end{aligned}\right.
\end{equation*}
in some small neighborhood of 0 in $\mathbb R^m$.
\end{theorem}

 This gives  ample smooth solutions for the pair  $(\tilde a, b) $ in \eqref{ma1}   near a neighborhood of $0$, whose jets up to order 1 at $0$ can be  prescribed arbitrarily. Consequently,  a rescaling method can be used to obtain  solutions on any bounded domain. This is due to the special structure of the equation \eqref{mae}, and the fact that  if  $u$ is a solution  on $B_r$ for some $r>0$, then $\tilde u(z, w): = \frac{R^2}{r^2}u(rz/R, rw/R)$ is a solution on $B_R$ for any $R>0$.

Alternatively we present another approach to obtaining infinitely many solutions to \eqref{ma1} that do  not rely on the rescaling process, but instead by making use of   an  existence theorem \cite[Theorem 1.6]{PZ} on large domains  as follows. 

\begin{theorem}\cite[Theorem 1.6]{PZ}\label{pz}
Let  $F =(F_1,\ldots, F_N) $ be a $C^2$ vector-valued function of  variables $(X, Y)\in \mathbb R^{N}\times \mathbb R^{mN}$, with $F(0)=0$ and $ \nabla F(0)=0$. For any $R>0$ and $0<\alpha <1$, there exist infinitely many $C^{2, \alpha}$ solutions $v = (v_1, \ldots, v_N)$ to the partial   differential system
    $$ \Delta v = F(v, \nabla v)\ \ \text{on}\ \ B_R. $$
 Moreover, all these solutions are of vanishing order 2 at $0\in \mathbb R^m$. Namely. $v(0)=0, \nabla v(0)=0, \nabla^2 v(0) \ne  0$. 
\end{theorem}

 It should be pointed out that, although \cite{PZ}   only focuses on real-valued systems, by splitting $v$ and $F$ into  real and imaginary parts accordingly,  Theorem \ref{pz1} and Theorem \ref{pz} can be easily applied to  complex-valued systems that we will be having.

  \begin{theorem}\label{mat}
   For each $R>0$, there exist infinitely many smooth strictly plurisubharmonic  functions  that satisfy \eqref{mae} on $  \mathbb C\times D_R$. All of these solutions  take on a special quadric form \eqref{smae}    in the first variable $z$. 
\end{theorem}

\begin{proof}

Take  $v= (\tilde a, b)$ in Theorem \ref{pz}, and $F(X_1, X_2, Y_1, Y_2, Y_3, Y_4) = (16e^{-2X_1}|Y_4|^2, 8Y_1Y_4) $. One can check that 
  $$ \Delta (\tilde a, b) = F(\tilde a, b, \tilde a_w, b_w, \tilde a_{\bar w}, b_{\bar w}),$$
  with $F(0)=0$ and $ \nabla F(0)=0$. Theorem \ref{pz} thus gives rise to infinitely many  $C^{2, \alpha}$ solutions $(\tilde a, b) $ on $D_R$ which are of vanishing order 2 at $0$.  
  By a standard bootstrapping argument, these solutions are necessarily smooth on $D_R$.

  Now, with each such  pair of solution  $(\tilde a, b) $,  let  $c$ be any holomorphic function. Then $c$ automatically satisfies
   \begin{equation}\label{c}
       c_{\bar w w} =  \tilde a_{   w} c_{\bar w} + 2 e^{-\tilde a} b_{\bar w}\bar c_{ w}\ \ \text{on}\ \ D_R.
   \end{equation} 
Finally, substituting  such $(\tilde a, b, c) $ into  the linear   equation: 
\begin{equation}\label{d}
    d_{\bar w w} =e^{-\tilde a}\left(|c_{\bar w}|^2+1\right) \ \ \text{on}\ \ D_R 
\end{equation} 
 to solve for a smooth $d$ on $D_R$. Altogether, we have obtained $a, b, c$ and $ d$ such that the system \eqref{ma0} is satisfied, and thus $u$ of the form \eqref{smae} with these coefficients provides infinitely many solutions to \eqref{mae}. That the solutions are strictly plurisubharmonic  is because we are  dealing with  $2\times 2$ Hessians.    
\end{proof}
  
  In particular, the theorem generates infinitely many smooth K\"ahler metrics $ \partial \bar\partial u$   whose   volume forms are constant. In fact, for any smooth positive function $f$ on $\mathbb C\times D_R$, the similar approach as in the proof  can be applied to produce infinitely many smooth strictly plurisubharmonic  functions defined on $  \mathbb C\times D_R$ that satisfy
$$  \det (\partial\bar\partial u) =f\ \ \text{on}\ \  \mathbb C\times D_R, $$

\section{An obstruction to a flat   metric}
In this section we   prove Theorem \ref{obp}:  an obstruction for a solution to \eqref{mae}   of the form \eqref{smae} to induce a flat K\"ahler metric. As an application, this leads to the construction of  a large class of solutions, not obtained via rescaling, to \eqref{mae} such that the associated   K\"ahler metrics are nowhere flat on cylindrical domain $\mathbb C\times D_R$  for any $R>0$. 

\begin{proof}[Proof of Theorem \ref{obp}: ]
 We first compute  the  K\"ahler metric $g =\partial \bar\partial u$ for a solution $u$ of the form  \eqref{smae}.  
\begin{equation}\label{g}
\begin{split}
    g =      \begin{bmatrix}
g_{1\bar 1  } &  g_{1\bar 2}\\
g_{ 2\bar 1}  &    g_{ 2 \bar 2}
\end{bmatrix}:   =\begin{bmatrix}
 a & \ \ a_{\bar w} \bar z + 2b_{\bar w}z + c_{\bar w} \\
a_{   w} z +2\bar b_{ w}\bar z + \bar c_{ w} & \ \ a_{\bar w w}|z|^2 + b_{\bar w w}z^2 +\overline{b_{\bar w w} z^2}+c_{\bar w w} z+\overline{c_{\bar w w}z }  +d_{\bar w w}
\end{bmatrix}. 
\end{split}
   \end{equation}
Since $\det g =1$, the inverse matrix $g^{-1}$ of $g$ is 
\begin{equation*}
\begin{split}
    g^{-1} =      \begin{bmatrix}
g^{\bar 1  1 } &  g^{\bar 1 2}\\
g^{ \bar 2 1}  &    g^{ \bar 2  2}
\end{bmatrix}  =  \begin{bmatrix}
a_{\bar w w}|z|^2 + b_{\bar w w}z^2 +\overline{b_{\bar w w} z^2}+c_{\bar w w} z+\overline{c_{\bar w w}z }  +d_{\bar w w} & \ \ -a_{\bar w} \bar z -2b_{\bar w}z - c_{\bar w} \\
-a_{   w} z -2\bar b_{ w}\bar z - \bar c_{ w} & \ \ a
\end{bmatrix}. 
\end{split}
   \end{equation*}

According to the general formulae for Christoffel symbols under a K\"ahler metric $g$ (see, for instance, \cite{Ba}\cite{Jos}):
 \begin{equation}\label{cs}
     \Gamma^\alpha_{\beta\gamma} = \frac{\partial g_{\gamma\bar \nu}}{\partial z_\beta}g^{\bar \nu \alpha}, \alpha, \beta, \gamma \in \{1, \ldots, n\}, 
 \end{equation}
 a direct computation gives  
  \begin{equation*}
     \begin{split}
        \Gamma_{11}^1  = \frac{\partial g_{1\bar 1}}{\partial z}g^{\bar 1  1} + \frac{\partial g_{1\bar 2}}{\partial z}g^{\bar 2  1} = &  -2b_{\bar w}\left(a_{   w} z +2\bar b_{ w}\bar z + \bar c_{ w} \right).        \end{split}
 \end{equation*}
Since  the K\"ahler metric is flat,    the curvature tensor $R\equiv 0$. In particular,  
\begin{equation}\label{ob}
    R_{1\bar 11}^1 = -\frac{\partial \Gamma_{11}^1}{\partial \bar z} =  4|b_{\bar w}|^2 \equiv 0.  
\end{equation} 
Namely, $b$ must be holomorphic on $D$.

Conversely, if $b_{\bar w}\equiv 0$ on $D$, then by \eqref{ma1}  
 \begin{equation}\label{ma2}
   \left\{  \begin{aligned}
      &\tilde a_{\bar w w} =  0;\\
      &c_{\bar w w} =  \tilde a_{   w} c_{\bar w}; \\
       &d_{\bar w w} =e^{-\tilde a}\left(|c_{\bar w}|^2+1\right).
    \end{aligned}\right.
\end{equation}
From the first equation above, there exists some holomorphic function $h$ on $D$ such that \begin{equation}\label{a2}
    \tilde a = 2\text{Re}(h),\ \ \text{ and so}\ \  a = e^{2\text{Re}(h)}. 
\end{equation} Letting $f: = e^{- \tilde a} \bar c_{ w} $,    from   the second equation in \eqref{ma2}, one can see that 
$$ f_{\bar w} =  e^{- \tilde a} \bar c_{\bar ww} - e^{- \tilde a} \tilde a_{\bar w} \bar c_{ w}=  e^{- \tilde a} (\overline{ c_{\bar ww} - \tilde a_w c_{\bar w}   }) =  0. $$
Namely, $f$ is holomorphic on $D$ and \begin{equation}\label{c2}
    c_{\bar w} = e^{2\text{Re}(h)}\bar f.
\end{equation} 
Finally, plugging \eqref{c2} into the last equation in \eqref{ma2} we have 
\begin{equation}\label{d2}
    d_{\bar w w} =  e^{2\text{Re}(h)}|f|^2+ e^{-2\text{Re}(h)}.
\end{equation} 
  
    Next we verify that  the   curvature tensor with respect to the  K\"ahler metric $g =\partial \bar\partial u$ is zero for such solutions with $b_{\bar w}=0$. Substituting \eqref{a2}-\eqref{d2} into \eqref{g},   the corresponding K\"ahler metric becomes
 \begin{equation*}
\begin{split}
    g =      \begin{bmatrix}
g_{1\bar 1  } &  g_{1\bar 2}\\
g_{ 2\bar 1}  &    g_{ 2 \bar 2}
\end{bmatrix}:   =\begin{bmatrix}
 e^{h}e^{\bar h} & \ \  e^he^{\bar h}\overline{h'}\bar z  +    e^{h}e^{\bar h}\bar f \\
e^he^{\bar h} {h'}  z  +  e^{h}e^{\bar h}  f & \ \ e^he^{\bar h}|h'|^2|z|^2  +   e^{h}e^{\bar h}h'\bar f z+e^{h}e^{\bar h}\overline{ h'} f \bar z   +  e^{h}e^{\bar h}|f|^2+ e^{-h}e^{-\bar h}
\end{bmatrix}. 
\end{split}
   \end{equation*}
Then the inverse matrix $g^{-1}$ of $g$ is 
\begin{equation*}
\begin{split}
    g^{-1} =      \begin{bmatrix}
g^{\bar 1  1 } &  g^{\bar 1 2}\\
g^{ \bar 2 1}  &    g^{ \bar 2  2}
\end{bmatrix}  =  \begin{bmatrix}
e^he^{\bar h}|h'|^2|z|^2  +   e^{h}e^{\bar h}h'\bar f z+e^{h}e^{\bar h}\overline{ h'} f \bar z  + e^{h}e^{\bar h}|f|^2+ e^{-h}e^{-\bar h} & \ \  -e^he^{\bar h}\overline{h'}\bar z  -    e^{h}e^{\bar h}\bar f \\
-e^he^{\bar h} {h'}  z -  e^{h}e^{\bar h}  f & \ \ e^{h}e^{\bar h}
\end{bmatrix}. 
\end{split}
   \end{equation*}
Here for any holomorphic function $g$ of the variable $w$, we use the notation $g':= \frac{\partial g}{\partial w}$.

 Making use of the formula \eqref{cs} we compute  all the Christoffel symbols as follows.
   \begin{equation*}
     \begin{split}
        \Gamma_{11}^1  = &\frac{\partial g_{1\bar 1}}{\partial z}g^{\bar 1  1} + \frac{\partial g_{1\bar 2}}{\partial z}g^{\bar 2  1} =  0;\\
        \Gamma_{12}^1  =  & \frac{\partial g_{2\bar 1}}{\partial z}g^{\bar 1  1} + \frac{\partial g_{2\bar 2}}{\partial z}g^{\bar 2  1} =  e^he^{\bar h} {h'}\left( e^he^{\bar h}|h'|^2|z|^2 +   e^{h}e^{\bar h}h'\bar f z+e^{h}e^{\bar h}\bar h' f \bar z  + e^{h}e^{\bar h}|f|^2+ e^{-h}e^{-\bar h} \right) +\\
        &+e^he^{\bar h}h'\left(\overline{ h'} \bar z+ \bar f\right)\left( - e^he^{\bar h} {h'}  z -  e^{h}e^{\bar h}  f\right) = h';\\
        \Gamma_{21}^1 = &\frac{\partial g_{1\bar 1}}{\partial w}g^{\bar 1  1} + \frac{\partial g_{1\bar 2}}{\partial w}g^{\bar 2  1} =  \Gamma_{21}^1 = h';\\
         \Gamma_{22}^1 =  & \frac{\partial g_{2\bar 1}}{\partial w}g^{\bar 1  1} + \frac{\partial g_{2\bar 2}}{\partial w}g^{\bar 2  1} = e^he^{\bar h}\left( (h')^2z    +  {h''}z +h'f+f'  \right) \left[e^he^{\bar h}(|h'|^2|z|^2+   h'\bar f z+\overline{ h'} f \bar z  + |f|^2)+ e^{-h}e^{-\bar h}\right] + \\
         &+\left[e^he^{\bar h}\left[(h'|h'|^2 +h''\overline{h'} )|z|^2+ ((h')^2+h'')\bar f z+ (|h'|^2f+ \overline{h'}f')\bar z +(h'|f|^2+ f'\bar f) \right] + e^{-h}e^{-\bar h}(-h')\right]\\
         &\left( - e^he^{\bar h} {h'}  z-  e^{h}e^{\bar h}  f\right) =  2(h'f+(h')^2z)+ h''z+f';\\
        \Gamma_{11}^2 =  &\frac{\partial g_{1\bar 1}}{\partial z}g^{\bar 1  2} + \frac{\partial g_{1\bar 2}}{\partial z}g^{\bar 2  2} = 0;\\
          \Gamma_{12}^2 =   &\frac{\partial g_{2\bar 1}}{\partial z}g^{\bar 1 2} + \frac{\partial g_{2\bar 2}}{\partial z}g^{\bar 2  2}= e^he^{\bar h} {h'} \left( -e^he^{\bar h}\overline{h'}\bar z  -    e^{h}e^{\bar h}\bar f \right) +  \left(e^he^{\bar h}|h'|^2\bar z+ e^{h}e^{\bar h}h'\bar f\right)e^he^{\bar h} =0;\\
         \Gamma_{21}^2 =  &\frac{\partial g_{1\bar 1}}{\partial w}g^{\bar 1  2} + \frac{\partial g_{1\bar 2}}{\partial w}g^{\bar 2  2}  =  \Gamma_{12}^2 = 0;\\
          \Gamma_{22}^2 =  &\frac{\partial g_{2\bar 1}}{\partial w}g^{\bar 1  2} + \frac{\partial g_{2\bar 2}}{\partial w}g^{\bar 2  2}  = e^he^{\bar h}\left( (h')^2z    +  {h''}z +h'f+f'  \right)e^he^{\bar h} \left(  -\overline{h'}\bar z  -   \bar f\right) +\\
          &+\left[e^he^{\bar h}\left[(h'|h'|^2 +h''\overline{h'} )|z|^2+ ((h')^2+h'')\bar f z+ (|h'|^2f+ \overline{h'}f')\bar z +(h'|f|^2+ f'\bar f) \right] + e^{-h}e^{-\bar h}(-h')\right]e^he^{\bar h}\\
          =&  -h'.
        \end{split}
 \end{equation*}
Specifically, all the Christoffel symbols are holomorphic.  Since for a K\"ahler metric, all the curvature tensor coefficients vanish, except 
 $$ R_{\alpha\bar\beta\gamma}^\delta = -\frac{\partial \Gamma_{\alpha\gamma}^\delta}{\partial \bar z_\beta}, R_{\bar\alpha\beta\gamma}^\delta = \frac{\partial \Gamma^\delta_{\beta\gamma}}{\partial \bar z_\alpha}, \ \ \alpha, \beta, \gamma\in \{1, \ldots, n\} $$
 and their conjugates, we have the curvature $R\equiv 0$. 
\end{proof}

\begin{theorem}\label{cf}
For any $R>0$, there exist infinitely many real-analytic plurisubharmonic  solutions   to \eqref{mae} on $\mathbb C\times D_R$ such that the corresponding K\"ahler metric is nowhere flat.  
\end{theorem}

\begin{proof}
    Given  any holomorphic function $f$ such that $|f|<1$  and $|f'|\ne 0$ on $D_R$, let 
$$\tilde a  = \ln \frac{|f'|}{ 1-|f|^2 }\ \ \text{on}\ \ D_R.$$ 
For example, one can choose $f$ to be any conformal map from $D_R$ into $D_1$. Apparently,  $$a= e^{\tilde a} = \frac{|f'|}{1-|f|^2}$$ is real-analytic on $D_R$.  

Let  $b$ solve  
$$   b_{\bar w}= \frac{|f'|^2}{2(1-|f|^2)^2} \ \ \text{on}\ \ D_R. $$ 
Indeed, since the right hand side has a Taylor expansion at $0$ with radius of convergence $R$, one can solve a $b$ by term-by-term integration of the  expansion at $0$ using the following general formula: 
$$ b = \sum_{i+j=0}^\infty \frac{b_{ij}}{j+1}w^{i}\bar w^{j+1} \ \ \text{solves}\ \  b_{\bar w} = \sum_{i+j=0}^\infty b_{ij}w^i\bar w^j $$
on its domain of convergence. Consequently, the Taylor expansion of $b$ has radius of convergence $R$. 
Moreover, one can directly verify that the pair $(\tilde a, b)$ satisfies the first two equations in \eqref{ma1}. Then as in the proof of Theorem \ref{mat}, we further let $c$ be any holomorphic function so \eqref{c} is satisfied, and  solve \eqref{d} for the Taylor expansion of  $d$ on $D_R$ (again, by term-by-term integration). Altogether, the function $u$ defined in \eqref{smae} with these coefficients  $a, b, c$ and $ d$  is a real-analytic    plurisubharmonic  solution to \eqref{mae} on $\mathbb C\times D_R$. With such chosen $u$, the corresponding K\"ahler metric is nowhere flat according to (the proof of) Theorem \ref{obp}.  
\end{proof}

Following the proof of Theorem \ref{cf}, let us produce a concrete real-analytic solution to \eqref{mae} on $\mathbb C\times D_R$. 
\begin{example}\label{ex}
Letting $f(w)= w/R$ on $D_R$, then $f'\ne 0, |f|<1$ on $D_R$. Following  the  proof of Theorem  \ref{cf}, we have 
   $$a = e^{\tilde a} = \frac{R}{R^2-|w|^2}, \ \ b_{\bar w} = \frac{R^2}{ 2(R^2-|w|^2)^2}.$$
In particular, in terms of Taylor expansion at $w=0$, $b_{\bar w} = \sum_{n=0}^\infty \frac{(n+1)|w|^{2n}}{ 2R^{2n+2}} $. Thus the term-by-term integration gives a solution $$b = \sum_{n=0}^\infty \frac{w^{n}\bar w^{n+1}}{2R^{2n+2}} = \frac{\bar w}{ 2( R^2-|w|^2) }.$$
Further letting   $c = 0 $ so  that \eqref{c} is automatically satisfied. Finally we solve $d$  so that \eqref{d} holds, which in this case becomes  
$$  d_{\bar ww} = e^{-\tilde a} = \frac{R^2-|w|^2}{R}. $$
One such $d$ is $$d =R|w|^2 - \frac{|w|^4}{4R}.$$ Altogether, we obtain a solution to \eqref{mae} on $\mathbb C\times D_R$:
\begin{equation}\label{co}
    u(z, w)  = \frac{R|z|^2}{R^2-|w|^2} +  \frac{\bar wz^2}{2( R^2-|w|^2)}+ \frac{  w\bar z^2}{ 2(R^2-|w|^2) }  + R|w|^2 - \frac{|w|^4}{4R}, 
\end{equation} 
whose induced K\"ahler metric is nowhere flat since $R_{111}^1 = 4|b_{\bar w}|^2\ne 0$ on $\mathbb C\times D_R$.

It is worth noting that   the metric given by the complex Hessian of $u$ in \eqref{co} is not complete. Indeed, the path $\gamma(t): = (0, t), 0<t<R$  originates from the origin with the initial velocity $(0, 1)$, and  approaches the boundary point $(0,R)$. However, the length of  $\gamma $ with respect to $g: =\partial\bar\partial u$ is 
$$ \int_0^R \|\dot{\gamma}\|_gdt = \int_0^R\sqrt{2 g_{\bar ww}|_{\gamma}}dt =  \int_0^R\sqrt{2\left(R-\frac{t^2}{R}\right) }dt= \frac{\sqrt{2}}{4} \pi R^{\frac{3}{2}} <\infty.$$
By the Hopf-Rinow Theorem, it is not complete. 

\end{example}

In particular, Calabi's question fails for  any bounded domain in $\mathbb C^2$, by simply restricting the solutions in Theorem \ref{cf} on this domain. 

\begin{cor}
   For every bounded domain $\Omega\subset \mathbb C^2$, there exist    infinitely many real-analytic plurisubharmonic  solutions   to \eqref{mae} on $\Omega$ such that its induced K\"ahler metric is nowhere flat.  
\end{cor}

\section{Entire solutions with flat K\"ahler metrics}

In this section, we convert the attention to entire solutions. First we  construct ample entire real-analytic plurisubharmonic solutions to \eqref{mae} on $\mathbb C^2$ of the quadratic form \eqref{smae}. 

\begin{theorem}\label{mate}
Given any three entire holomorphic functions $h$, $f$ and $b$  in $w\in \mathbb C$, define $$a: = e^{2\text{Re}(h)  }$$ and let $c$ and $d$ be solutions to  
\begin{equation}\label{7}
    c_{\bar w} = a\bar f, \ \ d_{\bar w  w} =  a|f|^2 +\frac{1}{a} \ \ \text{on}\ \ \mathbb C. 
\end{equation}  Then the function $u$ defined in \eqref{smae} with the coefficients  $(a, b, c, d)$  given above is    an entire real-analytic plurisubharmonic  solution to \eqref{mae} on $\mathbb C^2$. Moreover, the induced K\"ahler metric  given by the solution is flat. 
\end{theorem}

\begin{proof}
Since $h$ is holomorphic on $\mathbb C$, the function $a$ admits a Taylor expansion at $0$  with infinite radius of convergence. Consequently, both $a\bar f$ and $a|f|^2 +\frac{1}{a} $ also have    Taylor expansions at $0$  with infinite radius of convergence. Therefore, \eqref{7}   can be integrated    term-by-term using these   expansions, producing  Taylor expansions of $c$ and $d$  with infinite radius of convergence. This establishes the existence  of the coefficients $c$ and $d$ on $\mathbb C$. The fact that  these particular choices yield a solution to \eqref{mae} with a flat induced K\"ahler metric follows directly from the second part (when $b_{\bar w}\equiv 0$) in the proof to Theorem \ref{obp}. 
\end{proof}

In particular, this theorem generalizes the examples by Warren \cite{Wa}, and Myga \cite[Proposition 1]{My}  where $h $ is holomorphic, and $b=c\equiv 0$, with flat induced  K\"ahler metrics. That is,  none of our examples  yields a counterexample  to Calabi' s question.

\begin{example}
    Letting $ h=w, f= e^{-w}$, and $b$ be any holomorphic function (say, $b\equiv 0)$ in Theorem \ref{mate}, one can see that 
    $$ u(z, w) = e^{2\text{Re}(w)}|z|^2 +   \bar we^{w} z + we^{\bar w}\bar z + e^{-2\text{Re}(w)} +|w|^2$$
    is an entire solution to \eqref{mae}. 
\end{example}

Theorem \ref{mate} has found a special class of entire solutions  to \eqref{mae} of the form \eqref{smae} where  Calabi's open question is true.   A natural question arises whether  Calabi's  question is true for  any entire solutions  of the form \eqref{smae}. Namely, 
 
\medskip

\noindent \textbf{Question:} Is  the K\"ahler metric given by the complex Hessian of any smooth solution to \eqref{mae} on $\mathbb C^2$ of the form \eqref{smae} necessarily flat? By Theorem \ref{obp}, this is further equivalent to asking: must every    smooth entire solution of the form \eqref{smae} satisfy the condition $b_{\bar w}\equiv 0$? 

\medskip

Although it remains unclear for the above question in full generality, we show that $b_{\bar w}$ must have zeros somewhere on $\mathbb C$, 
as a step toward addressing it. The proof  makes use of a non-existence result of entire solutions by Osserman \cite{Os}.

\begin{theorem}\label{cal}
For every  $C^2$-smooth entire solution  to \eqref{mae} on $\mathbb C^2$ of the form \eqref{smae},  the zero set $Z_b: =\{w\in \mathbb C:  b_{\bar w}(w)=0\}$ of $b_{\bar w}$ must be nonempty. 
\end{theorem}

\begin{proof}
Assume by contradiction that there exists a $C^2$  solution   of the form \eqref{smae} such that 
$b_{\bar w}$ is nowhere zero on $\mathbb C$.  Let us revisit the original system \eqref{ma1}. Letting $h: = e^{-2\tilde a}{b_{\bar w}}$ and making use of the equation $b_{\bar w w} = 2\tilde a_wb_{\bar w} $, one has $h$ satisfies
$$  h_w = e^{-2\tilde a}{b_{\bar ww}} - 2\tilde a_w e^{-2\tilde a}b_{\bar w} = e^{-2\tilde a} 2\tilde a_wb_{\bar w} - 2\tilde a_w e^{-2\tilde a}b_{\bar w} =0\ \ \text{on}\ \ \mathbb C.$$
Namely, $h$ is    anti-holomorphic on $\mathbb C$, with  \begin{equation}\label{bh}
b_{\bar w} = e^{2\tilde a}h\ \ \text{on}\ \ \mathbb C. \end{equation}
Moreover, since $b_{\bar w}\ne 0$,   $ h$ is nowhere zero on $\mathbb C$.


Plugging it to the equation for $\tilde a$, we obtain
 \begin{equation}\label{os}
     \tilde a_{\bar w w} = 4|h|^2e^{2\tilde a}\ \ \text{on}\ \ \mathbb C. 
 \end{equation}
We claim there is no entire solution to \eqref{os}. If not, since $h$  is nowhere zero on $\mathbb C$,   $\ln (4|h|) $ is harmonic on $\mathbb C$, and thus $v: =  \ln (4|h|)  + \tilde a $ is a  $C^2$ solution to
\begin{equation}\label{os1}
  \Delta   v = e^{2v} \ \ \text{on}\ \ \mathbb C. 
\end{equation}
However, we recall   a nonexistence result of Osserman \cite{Os} for solutions to  $ \Delta u = f(u)$ on $\mathbb C$: if $f>0, f'\ge 0$ on $\mathbb R$ and  
 $$  \int_{0}^\infty\left(\int_0^t f(s)ds\right)^{-\frac{1}{2}}dt <\infty,$$ 
 then there is no $C^2$ solution to $ \Delta u = f(u)$ on $\mathbb C$.
One can check that   $f(s): = e^{2s}$ in our case satisfies the above assumptions for $f$.  From this we immediately obtain the nonexistence of entire solutions to \eqref{os1}. But this is  a contradiction! The claim is thus proved, and so is the theorem. 
\end{proof}

Unfortunately the approach used in the proof of Theorem \ref{cal} no longer works  if the function $h$ there has   zeros. This is because  the singularity of $v$ cannot  be resolved at the zeros of $h$ to yield  a (weak) solution to \eqref{os1} everywhere.    In the Appendix we will discuss several   cases where the isolated singularities can, in fact, be resolved.     On the other hand, note that due to the holomorphy of $h$, the zero set $Z_b$ of $b_{\bar w}$, if not empty, is either isolated or the whole $\mathbb C$ from \eqref{bh}. Therefore    to answer \textbf{Question}, by Theorem \ref{cal} it suffices to consider the case when $Z_{b}$ is isolated.

\section{Solutions that are radially symmetric  on  $w$}

In this section, we investigate  solutions to \eqref{mae} that depends quadratically on the variable $z$ and  exhibit radial symmetry in the  variable $w$. More precisely, the solutions are of the form \eqref{smae}, where $a, b, c$ and $ d$ depend solely on  $|w|$. In the following, we   derive   explicit formulas for all these functions. Interestingly, such solutions exhibit a wide range of  singular behaviors.
\medskip

Letting $t: = \log |w|^2$, we shall explore all possible forms of expressions of $a, b, c$ and $ d$ in terms of the variable $t$.  For a  radial function $h $ of the variable $w$, by the chain rule,  $h_{w} = \frac{h_t}{  w}$, $h_{\bar w} = \frac{h_t}{\bar w}$, and $h_{\bar ww} = \frac{h_{tt}}{|w|^2}$  away from $w=0$. Here $h_t$ denotes the derivative of $h$ with respect to $t$. In view of this,   \eqref{ma0} can be reduced to a system of ordinary differential equations:
\begin{equation} \label{re}
   \left\{  \begin{aligned}
       &a a_{tt} = |a_t|^2 + 4|b_{t}|^2;\\
       &ab_{tt} = 2a_tb_{t};\\
       &ac_{tt} = a_tc_t+2b_t\bar c_t;\\
       &ad_{tt} = |c_t|^2+ e^t.
    \end{aligned}\right.
\end{equation}

To  solve the system, let $\tilde b: = b_{t}$, so the second equation in \eqref{re} becomes 
$$ a\tilde b_t = 2a_t\tilde b. $$
Separating the two functions and then taking integration on both sides, one gets 
\begin{equation}\label{tb}
    \tilde b \ (=  b_{t}) =  ka^2 
\end{equation} 
for any complex constant $k$. 

Plugging \eqref{tb} into the first equation in \eqref{re}, we  seek  solutions to  
\begin{equation}\label{ar}
     a a_{tt} - a_t^2 = 4 |k|^2a^4.
\end{equation} 
  Let \( v: = \frac{a_t}{a} \), or equivalently
  \begin{equation}\label{ar1}
      a_t = av. 
  \end{equation} Then $
v_t =\frac{a a_{tt} - a_t^2}{a^2}$, and  
    \eqref{ar} can be simplified to a first-order ODE in terms of \( v \):
\begin{equation}\label{ar2}
    v_t= 4 |k|^2a^2
\end{equation}
Now that  a coupled system \eqref{ar1}-\eqref{ar2} is  established about $(a, v)$,   by separating  the variables, we get
 $$ v\ dv = 4|k|^2a\ da. $$
Integrating both sides, one can eventually obtain 
$$ 
v^2 = 4 |k|^2a^2 + C_1,
$$
for any real constant  \(C_1\) (since $a$ and  $v$ are real) such that the right hand side is nonnegative. Recalling that $ v = \frac{a_t}{a} $, hence 
\begin{equation} \label{att}
    a_t  = \pm a\sqrt{4 |k|^2a^2 + C_1}  
\end{equation}
for any real constant $C_1$ with $4 |k|^2a^2 + C_1\ge 0 $, and for the constant $k$ chosen in \eqref{tb}. Depending on the value  of $k$, and then that of $C_1$, we explore all possible expressions of $a, b, c$ and $d$ below.
\medskip

\textbf{Case I: $k=0$.}   A straight forward computation from \eqref{att} and \eqref{tb} gives  
\begin{equation}\label{ak0}
    a = C_1e^{C_2t}, \ \ b = C_3
\end{equation}
for any real constants $C_1>0$ (since  $a>0$)  and $ C_2$,  and any complex constant $C_3$. Plugging them into the third and fourth equations in \eqref{re}, we have
\begin{equation}\label{ak01}
\begin{split}
    c&= C_4e^{C_2t}+C_5,  \\
    d&= \begin{cases}
        \dfrac{|C_4|^2e^{C_2t}}{C_1} +\dfrac{e^{(1-C_2)t}}{(1-C_2)^2C_1} + C_6t+C_7, &  C_2\ne 1,\\
         \dfrac{|C_4|^2e^{C_2t}}{C_1} + \dfrac{t^2}{2C_1}+C_6t+C_7, & C_2 = 1,
    \end{cases}  
\end{split}
\end{equation}
where $C_1>0$, $C_2, C_6$ and $C_7$ are any real constants, and $C_3, C_4$ and $C_5$ are any complex constants.

\medskip

\textbf{Case II: $k\ne 0$.}  Rewrite  \eqref{att} to be
\begin{equation}\label{se}
    a_t  = \pm 2|k| a\sqrt{a^2 + C_1},
\end{equation}
where $C_1$ is any real constant such that $a^2+C\ge 0$. Separate the variables in \eqref{se} and then integrate both sides to get
 \begin{equation}\label{im}
     \pm \int \frac{da}{a\sqrt{a^2 + C_1}} = 2|k|t+C   
 \end{equation} 
 for any real constant  $C$. Depending on the sign of $C_1$,  there are three cases   for the integral on the left hand side of \eqref{im}:
 \begin{equation*}
 \begin{split}
     \int \frac{da}{a\sqrt{a^2 + C_1}} =  \left\{  \begin{aligned} & \frac{1}{\sqrt{C_1}}\ln\left(\frac{\sqrt{a^2+C_1} -\sqrt{C_1}}{a}\right)+C, & \ \ \ C_1>0;\\
     & -\frac{1}{a} +C, &\ \ \ C_1=0;\\
   &     \frac{1}{\sqrt{-C_1}}\sec^{-1}\left(\frac{a}{\sqrt{-C_1}}\right) +C, &\ \ \ C_1<0.
    \end{aligned} \right.
 \end{split}
 \end{equation*}
 Combining these formulas with \eqref{im}, 
we can solve $a$ explicitly:  
\begin{equation}\label{ae}
\begin{split}
         a \ (\text{or}\ -a) = \begin{cases}
        \dfrac{2\sqrt{C_1}C_2e^{2 |k|\sqrt{C_1}t}}{1-C_2^2e^{  4 |k|\sqrt{C_1}t}}, \ \ \ &C_1>0;\   \\
        -\dfrac{1}{2 |k|t + C_2}, \ \ \ &C_1=0; \\
         \sqrt{-C_1}  \sec\left( 2|k|\sqrt{-C_1} t + C_2  \right), \ \ \ &C_1< 0,
    \end{cases}
    \end{split}
\end{equation}
where  $C_2$ is an arbitrary real constant. The domain of definition for $a$ is wherever   the expression on the  right hand side is defined  such that  $a>0$. 

Next we solve for $b$. Recalling that $ b_t =   k a^2 $ and combining it with the expression \eqref{ae} of $a$,  one can immediately obtain    
\begin{equation}\label{br}
      b = \begin{cases}
      \dfrac{\sqrt{C_1}k}{|k|(1 - C_2^2 e^{4 |k|\sqrt{C_1} t})} + k_0, &  C_1>0; \\
    - \dfrac{k}{2|k|(2 |k|t + C_2)} + k_0, & C_1=0; \\
      \dfrac{k\sqrt{-C_1}}{2|k|} \tan\left(2 |k|\sqrt{-C_1}   t +  C_2\right) + k_0, & C_1<0,
\end{cases}
\end{equation}
 where  $C_2$ is the  real constant chosen in \eqref{ae}, and $k_0$ is an arbitrary complex constant.

 Plugging expressions of $a$ and $b$ (in particular, \eqref{tb} and \eqref{se}) into the third   equation  in \eqref{re} for $c$, we further solve
 $$  c_{tt} = \pm 2|k|\sqrt{a^2 + C_1}  c_t+2  ka \bar c_t.  $$
  Writing $c = c_1+ic_2$, where $c_1, c_2$ are real functions, and $k = k_1+ik_2$, where $k_1, k_2$ are real constants,  and separating the real parts from the imaginary parts in the above expression, we get 
 \begin{equation}\label{coup}
     \begin{split}
      \begin{cases}   (c_1)_{tt} = \pm 2|k|\sqrt{a^2 + C_1}  (c_1)_t+2 a\left( k_1  (c_1)_t + k_2(c_2)_t \right);\\
         (c_2)_{tt} = \pm 2|k|\sqrt{a^2 + C_1}  (c_2)_t+2 a\left( k_2  (c_1)_t - k_1(c_2)_t \right). \end{cases}
     \end{split} 
 \end{equation}
 This can be rewritten in a matrix form as follows:
\[
  \begin{pmatrix} (c_1)_{tt} \\ (c_2)_{tt} \end{pmatrix} = A   \begin{pmatrix} (c_1)_t \\ (c_2)_t \end{pmatrix},
\]
where
\[
A  = \pm2|k| \sqrt{a^2 + C_1}   I_2 + 2 a K, \ \text{with}\ \quad I_2 = \begin{pmatrix} 1 &  0\\ 0 & 1 \end{pmatrix} \ \text{and}\ \quad K = \begin{pmatrix} k_1 & k_2 \\ k_2 & -k_1 \end{pmatrix}.
\]
Note that the constant matrix \( K \) has eigenvalues \( \lambda = \pm \sqrt{k_1^2 + k_2^2} =\pm |k| \). So there exists a constant unitary matrix \( P \)   that  diagonalizes \( K \):
\[
K = P \begin{pmatrix} |k|  & 0 \\ 0 & -|k| \end{pmatrix} P^{-1}.
\]
Defining the  new variables \( \begin{pmatrix} w_1 \\ w_2 \end{pmatrix} = P^{-1} \begin{pmatrix} c_1 \\ c_2 \end{pmatrix} \), the system can be  decoupled into:
\[
\begin{cases}
(w_1)_{tt} = 2|k|\left(\pm\sqrt{a^2 + C_1} +   a  \right) (w_1)_t; \\
(w_2)_{tt} = 2|k|\left(\pm\sqrt{a^2 + C_1} -   a   \right) (w_2)_t.
\end{cases}
\]
Thus, without loss of generality, let us assume   for  simplicity that  $k$ is a positive constant and the "$+$" sign  in \eqref{coup} is taken. It then can be   rewritten as  
 \begin{equation}\label{ctt}
     \left\{\begin{aligned}
         (\ln |(c_1)_t|)_t &=  2k(\sqrt{a^2 + C_1}  +a);\\
         (\ln |(c_2)_t|)_t &=  2k(\sqrt{a^2 + C_1}  - a).
     \end{aligned}\right.
 \end{equation} 

Making use of \eqref{ae}, which further gives  
\begin{equation*}\label{ae1}
\begin{split}
        \sqrt{a^2+ C_1}  = \begin{cases}
        \dfrac{\sqrt{C_1}(1+ C_2^2e^{4 |k|\sqrt{C_1}t})}{1-C_2^2e^{  4 |k|\sqrt{C_1}t}}, 
        \ &C_1>0;\\
        -\dfrac{1}{2 |k|t + C_2},  \ & C_1= 0;\\
        \sqrt{-C_1} \tan\left( 2\sqrt{-C_1} |k|t + C_2  \right), \  & C_1< 0, 
    \end{cases}
    \end{split}
\end{equation*}
we  take integration directly on both sides of \eqref{ctt} to have
\begin{equation}\label{c1t} 
(c_1)_t =
\begin{cases}
    \displaystyle   \dfrac{C_{3, 1}e^{2k\sqrt{C_1}t}}{\left(1 - C_2 e^{2k\sqrt{C_1}t}\right)^2} ; \\ 
   \displaystyle  \dfrac{C_{3, 1}}{(2kt + C_2)^2} ;  \\ 
  \displaystyle    \dfrac{C_{3, 1}}{ 1 -\sin \left(2\sqrt{-C_1}|k|t+ C_2\right) },
\end{cases}  (c_2)_t =
\begin{cases}
    \displaystyle    \dfrac{C_{3, 2}e^{2k\sqrt{C_1}t}}{\left(1 + C_2 e^{2k\sqrt{C_1}t}\right)^2}  , & C_1 > 0; \\ 
  \displaystyle  C_{3, 2}, & C_1 = 0;  \\ 
  \displaystyle  \dfrac{C_{3, 2}}{ 1 +\sin \left(2\sqrt{-C_1}|k|t+ C_2\right) }, & C_1 < 0,
\end{cases}
\end{equation}
where $C_{3,1}, C_{3,2}$ are any real constants. 

Integrating one more time in \eqref{c1t}, we get
\begin{equation*} 
c_1 =
\begin{cases}
    \displaystyle   \dfrac{C_{3, 1}}{2k\sqrt{C_1}C_2(1 - C_2 e^{2k\sqrt{C_1}t})} + C_{4,1}, & C_1 > 0; \\ 
   \displaystyle  - \dfrac{C_{3, 1}}{2k(2kt + C_2)} +C_{4,1}, & C_1= 0;  \\ 
  \displaystyle  \frac{C_{3, 1}}{2\sqrt{-C_1}k} \left( \sec\left(2\sqrt{-C_1}kt+ C_2\right) + \tan\left(2\sqrt{-C_1}kt + C_2 \right) \right) + C_{4,1}, & C_1 < 0,
\end{cases}
\end{equation*}
and 
\begin{equation*} 
c_2 =
\begin{cases}
    \displaystyle  -  \dfrac{C_{3, 2}}{2k\sqrt{C_1}C_2 (1 + C_2 e^{2k\sqrt{C_1}t})} +C_{4, 2}, & C_1 > 0; \\ 
  \displaystyle  C_{3, 2}t + C_{4, 2}, & C_1 = 0;  \\ 
  \displaystyle  -\frac{C_{3, 2}}{2\sqrt{-C_1}k} \left( \sec\left(2\sqrt{-C_1}kt+ C_2\right)  - \tan\left(2\sqrt{-C_1}kt + C_2 \right)\right) + C_{4, 2}, & C_1 < 0,
\end{cases}
\end{equation*}
where $C_{4, 1}, C_{4, 2}$ are any real constants. 
Thus
\begin{equation}\label{cr}
c = c_1+i c_2 = 
\begin{cases}
    \displaystyle  \dfrac{ \overline{ C_3}+ C_3 C_2 e^{2k\sqrt{C_1}t}  }{2k\sqrt{C_1}C_2(1 - C_2^2 e^{4k\sqrt{C_1}t})}  + C_4 , & C_1 > 0; \\ 
   \displaystyle  -\dfrac{\text{Re}(C_3)}{2k(2kt + C_2)} +   i\text{Im}(C_3)t + C_4 , & C_1= 0;  \\ 
  \displaystyle \frac{\overline{ C_3}}{ 2\sqrt{-C_1}k} \sec\left(2\sqrt{-C_1}kt+ C_2\right) +\frac{C_3}{2\sqrt{-C_1}k} \tan\left(2\sqrt{-C_1}kt + C_2 \right) +C_4, & C_1 < 0, 
\end{cases}
\end{equation}
where $k>0$ and $C_2$ are defined in \eqref{ae} and \eqref{br}, and  $C_3: = C_{3,1}+ iC_{3,2}$ and $ C_4:  = C_{4,1}+iC_{4,2}$ are  arbitrary complex constants.

Finally, one can substitute the expression \eqref{ae} of  $a$, and \eqref{c1t} of $c_t$ into the fourth equation in \eqref{re} to solve for $d$:  
\begin{equation}\label{dr}
d = \begin{cases}
          \dfrac{e^{(1 - 2 k\sqrt{C_1})t}}{2\sqrt{C_1}C_2 (1 - 2 k\sqrt{C_1})^2} - \dfrac{C_2 e^{(1 + 2 k\sqrt{C_1})t}}{2\sqrt{C_1} (1 + 2 k\sqrt{C_1})^2}+\\
          +\dfrac{1}{8k^2 C_1^\frac{3}{2} C_2} \left( \dfrac{C_{3,1}^2}{1 - C_2 e^{2k\sqrt{C_1}t}}  
            -\dfrac{C_{3,2}^2}{1 + C_2 e^{2k\sqrt{C_1}t}}   \right)  + C_5 t + C_6, \ \ \ &C_1>0;\ \\
        - (2 k t + C_2 - 4 k) e^t   -\dfrac{C_{3,1}^2}{8k^2 (2kt + C_2)} - C_{3,2}^2\left(\dfrac{kt^3}{3} + \dfrac{C_2t^2}{2}\right)  + C_5 t + C_6,  &C_1 =0;\\
          \dfrac{e^t \left[(1+ 4k^2C_1 ) \cos\left(2 k\sqrt{-C_1}  t + C_2\right) + 4 k\sqrt{-C_1}  \sin\left(2 k \sqrt{-C_1} t + C_2\right) \right]}{\sqrt{-C_1}(1  - 4  k^2C_1)^2} +\\
      +  \dfrac{  
(C_{3,1}^2 - C_{3,2}^2) \tan\left(2k\sqrt{-C_1}t + C_2\right)
+ (C_{3,1}^2 + C_{3,2}^2) \sec\left(2k\sqrt{-C_1}t + C_2\right) }{4k^2C_1}  + C_5 t + C_6,  &C_1 =0,
    \end{cases}
 \end{equation}
for $k>0$,  $C_2$ and  $C_{3,1}, C_{3,2},  C_{4,1},  C_{4,2} $  defined in \eqref{ae}, \eqref{br} and \eqref{cr} with  $C_3 = C_{3,1}+ iC_{3,2}$ and $ C_4 = C_{4,1}+iC_{4,2}$, and any real constants $  C_5$ and $ C_6 $.

\medskip

Based on the above derivation for solutions to  \eqref{mae} that depend  radially on $w$, choosing different values for the parameters yields  the  following interesting examples with varying singularities. Note that the expressions for  $c$ do not introduce any additional singularities beyond those arising from $a$, $b$ or   $t = \ln|w|^2$.  For the purpose of studying singularity of solutions, we will consider the special cases when   $C_3  = C_4=0$, which significantly simplifies the expression for $d$.  We also let $k_0 = C_5 = C_6 =0$ below. 

\begin{example}\label{ex1}
    Taking $k  = 0$ and $C_1 =1$, $C_2\ne 1$  in \eqref{ak0} and \eqref{ak01}, we obtain a solution  to \eqref{mae}:  
\begin{equation*}
    u =    |w|^{2C_2}|z|^2 +  (z^2 +\bar z^2) +   \frac{|w|^{2-2C_2}}{(1-C_2)^2}. 
\end{equation*}
Different choices of $C_2$ lead to solutions with different regularity. For instance, if $C_2=0$, then $u$ is a (smooth) quadratic function
\begin{equation*}
    u =    |z|^2 +  (z^2 +\bar z^2) +    |w|^{2}; 
\end{equation*}
if $C_2 = \frac{3}{2}$, then $u$ has singularity at $w=0$: 
\begin{equation*}
    u =   |w|^3 |z|^2 +  (z^2 +\bar z^2) +    \frac{4}{|w|}; 
\end{equation*}
if $C_2 = \frac{1}{2}$, then $u$ is Lipschitz at $w=0$:
\begin{equation*}
    u =    |w||z|^2 +  (z^2 +\bar z^2) +    4|w|.
\end{equation*}
In particular, the last function  coincides with an example of B\l ocki \cite{Bl} and He \cite{He1} when $n=2$,   who showed it is both a pluripotential and viscosity solution.
\end{example}

\begin{example}
    Taking $k>0$ and $C_1= C_2= 1$   in \eqref{ae}, \eqref{br} and  \eqref{dr}, we obtain a solution  to \eqref{mae}:
$$ u=  \frac{2|w|^{4k}}{1-|w|^{8k}}|z|^2 + \frac{1}{1-|w|^{8|k|}}(z^2 +\bar z^2) +  \frac{|w|^{2-4k}}{2(1-2k)^2} - \frac{|w|^{2+4k}}{2(1+2k)^2}.  $$
Different choices of $k$ lead to solutions with different regularity. For instance, if $k=1$, then $u $ blows up at $w =0$ and $|w|=1$:
\begin{equation*}
    u = \frac{2|w|^4}{1-|w|^{8}}|z|^2 + \frac{1}{1-|w|^8}(z^2 +\bar z^2) +  \frac{1}{2|w|^2} - \frac{|w|^6}{18};\ 
\end{equation*}
if $k=\frac{1}{4}$, then $u $ is Lipschitz   at $w =0$, and  blows up at $|w|=1$:
\begin{equation*}
    u = \frac{2|w|}{1-|w|^2}|z|^2 + \frac{1}{1-|w|^2}(z^2 +\bar z^2) +   2|w| - \frac{2|w|^3}{9}, \ \ 
\end{equation*}
which is a pluripotential and viscosity solution on $\mathbb C\times D_1$, similar  as in \cite{Bl} and \cite{He1}.
\end{example}

\begin{example}
    Taking $k  = 1$ and $C_1 = C_2 =0$ in \eqref{ae}, \eqref{br} and \eqref{dr}, we obtain a solution  to \eqref{mae}:
\begin{equation*}
    u =  - \frac{1}{2\ln(|w|^2)} |z|^2 -\frac{1}{4\ln(|w|^2)}  (z^2 +\bar z^2) -\left(2\ln(|w|^2)-4\right)|w|^2\ \ \text{on}\ \ \mathbb C\times D_1.  
\end{equation*}
This is the  example of a pluripotential and viscosity solution  given by Wang-Wang \cite{WW}.  
This solution is in $W^{1, 2}_{loc}\cap W^{2,1}_{loc}$ but fails to be in $W^{1, p}_{loc}$ for any $p>2$, or in $W^{2, q}_{loc}$ for any $q>1$. Moreover, it is not even Dini continuous near $w= 0$. 
\end{example}

\begin{example}\label{ex2}
    Taking $k =-C_1 = 1$ and $C_2  =0$ in \eqref{ae}, \eqref{br} and  \eqref{dr}, we obtain a solution  to \eqref{mae}:
\begin{equation*}
    u =  \frac{1}{\cos(2\ln(|w|^2))} |z|^2 + \frac{\tan(2\ln(|w|^2))}{2}(z^2 +\bar z^2) +   \frac{|w|^2\left(-3\cos(2\ln(|w|^2)) + 4\sin(2\ln(|w|^2))\right)}{25}.
\end{equation*}
This solution  exhibits  singularities at $ |w| = e^\frac{\pi}{8}e^\frac{k\pi}{4}$ for each $ k\in \mathbb Z$. In particular, when $k\rightarrow -\infty$, the singularity set accumulates to $w=0$.  Moreover,  this solution is no longer plurisubharmonic   due to the frequent sign change. 
\end{example}

 We conclude the section  with an example of a solution to \eqref{mae} whose singularity set is of  real codimension-one.

\begin{example}
 It is straightforward to see that    \begin{equation}
        u(z, w) =\begin{cases}
            |z-1|^2+|w|^2, & \text{Re} (z )\le \frac{1}{2};\\
             |z|^2+|w|^2, & \text{Re} (z )> \frac{1}{2} 
        \end{cases}
    \end{equation}
  solves \eqref{mae} on $\mathbb C^2\setminus X$, where $X: =\{ (z, w)\in \mathbb C^2: \text{Re} (z ) = \frac{1}{2}\}$ is a real hypersurface of real codimension-one. Note that since $u$ is is plurisubharmonic and continuous everywhere,  the fundamental result  \cite{BT} of Bedford and Taylor guarantees that  $(dd^c u)^2$ is well-defined as a positive Borel measure on $\mathbb C^2$, where   $d= \bar\partial +\partial$, $d^c = \frac{i}{2}(\bar\partial-\partial)$. In detail, letting   $dV$ be the volume form on $\mathbb C^2$, and $dS$ be the surface measure on $X$, then the measure is $$    (dd^c u)^2 =  4dV + 2\lambda_X     \ \ \text{on}\ \ \mathbb C^2, $$
  where   $\lambda_X$ is a Lelong current over $X$ defined by
  $ \lambda_X(\phi)  =\int_X\phi dS  $ for any testing function $\phi$. In particular, the example shows that the singularity of $u$ at $X$ is not removable for \eqref{mae}.
  \end{example}

\section{Rigidity of more general solutions}
In the previous sections, we considered solutions dependent quadratically on the variable $z$. A natural question is to consider a more general form of solutions by, say, replacing $z$ in \eqref{smae} by a holomorphic function $\phi$ of $z$. That is,
\begin{equation}\label{6}
     u(z, w) = a(w) |\phi(z)|^2 + b(w) \phi^2(z) +\overline{b(w) \phi^2(z)}+c(w)\phi(z)+\overline{c(w)\phi(z)}  +d(w). 
\end{equation}
 The complex Hessian of $u$ in \eqref{6}  becomes 
\begin{equation}\label{ch}
    \partial \bar\partial u  =  \begin{bmatrix}
u_{\bar z z} &  u_{z\bar w}\\
u_{\bar z w}  &    u_{\bar w w}
\end{bmatrix}  =  \begin{bmatrix}
 a|\phi'|^2 & \ \ a_{\bar w}\phi'\bar \phi + 2b_{\bar w}\phi'\phi + c_{\bar w}\phi' \\
a_{ w}\overline{\phi'} \phi + 2\bar b_{ w}\overline{\phi\phi' }+ \bar c_{ w}\overline{\phi'} & \ \ a_{\bar w w}|\phi|^2 + b_{\bar w w}\phi^2 +\overline{b_{\bar w w} \phi^2}+c_{\bar w w} \phi+\overline{c_{\bar w w}\phi }  +d_{\bar w w}
\end{bmatrix}.
\end{equation}
 Due to the strict plurisubharmonicity of $u$, one has $\phi'\ne 0$ necessarily. By further adjusting the coefficients $(a, b, c, d)$, we can assume that $\phi$  is normalized such that $$\phi(0) =0\ \ \text{and}\ \ \phi'(0)=1.$$ In order for \eqref{6} to solve \eqref{mae}, the following theorem demonstrates that the choice  of  $\phi$ takes  very rigid forms. 

\begin{theorem}\label{gen}
    Let $u$ be a smooth solution to \eqref{mae} of the form \eqref{6} for some holomorphic function $\phi$ of $z$ such that $\phi(0) =0$, $\phi'\ne 0$ and $\phi'(0)=1$, then either $$\phi = z\ \ \text{on}\ \ \mathbb C,$$
    or there exists some nonzero constant $\alpha $ such that  \begin{equation}\label{5}
         \phi =   \frac{1-\sqrt{1-2\alpha z}}{\alpha}\ \ \text{on}\ \ D_{\frac{1}{2|\alpha|} },
    \end{equation} 
    where  the complex  square root takes the principal branch. 
   \end{theorem}

\begin{proof}
 Taking the determinant of  the Hessian \eqref{ch} and applying \eqref{mae}, we obtain 
\begin{equation}\label{dd}
\begin{split}
     |\phi'|^{-2} (=   |\phi'|^{-2} \det( \partial \bar\partial u) )= &|\phi|^2\left(a a_{\bar w w} -|a_w|^2 - 4|b_{\bar w}|^2 \right) + \phi^2\left(ab_{\bar w w}- 2a_wb_{\bar w}\right) +\overline{\phi^2} \left( a\overline{b_{\bar w w}}- 2a_{\bar w} \bar b_{ w}  \right)\\
        & + \phi\left(ac_{\bar w w} - a_{   w} c_{\bar w} - 2 b_{\bar w}\bar c_{ w} \right) +\bar{\phi}\left(a\bar c_{\bar ww} - a_{ \bar   w} \bar c_{w} - 2 \bar b_{ w}c_{\bar  w}   \right) + \left( ad_{\bar w w}-|c_{\bar w}|^2\right).
        \end{split}
\end{equation}
Write both sides in terms of  the Taylor expansion of $z$ at $0$. We first collect terms without $z$ on both sides by letting $z=0$. Recalling the normalization assumptions on $\phi$, this gives
\begin{equation}\label{aa}
    ad_{\bar w w}-|c_{\bar w}|^2 =1.
\end{equation}  
Collecting the coefficients of the term   $z$ in \eqref{dd}, one has
\begin{equation}\label{1}
ac_{\bar w w} - a_{   w} c_{\bar w} - 2 b_{\bar w}\bar c_{ w}  = -\phi''(0).
\end{equation}
Collecting the coefficients of the term   $z^2$ in \eqref{dd}, one has
$$     ab_{\bar w w}- 2a_wb_{\bar w}  + \frac{\phi''(0)}{2} ( ac_{\bar w w} - a_{   w} c_{\bar w} - 2 b_{\bar w}\bar c_{ w} ) =-  \frac{\phi'''(0)}{2} + (\phi''(0))^2. $$
Combined with \eqref{1}, one futher gets
\begin{equation}\label{bb}
ab_{\bar w w}- 2a_wb_{\bar w}   =  \frac{3(\phi''(0))^2- \phi'''(0)}{2}.  
\end{equation}
Collecting the coefficients of the term   $|z|^2$ in \eqref{dd}, one has
\begin{equation}\label{cc}
a a_{\bar w w} -|a_w|^2 - 4|b_{\bar w}|^2 = |\phi''(0)|^2. 
\end{equation}
 Substituting \eqref{aa}, \eqref{1}, \eqref{bb} and \eqref{cc} into \eqref{dd}, we obtain
\begin{equation}\label{4}
\begin{split}
     |\phi'|^{-2} =  |\phi''(0)|^2|\phi|^2  + \frac{3(\phi''(0))^2- \phi'''(0)}{2}\phi^2 +\overline{  \frac{3(\phi''(0))^2- \phi'''(0)}{2} \phi^2}   -  \phi''(0) \phi -\overline { \phi''(0)\phi}  + 1.
        \end{split}
\end{equation}

Denoting  $g: = (\phi')^{-1} $ and taking $\Delta$ on both sides of \eqref{4}, one gets
$   |g'|^{2} =  |\phi''(0)|^2|\phi'|^2 $, or equivalently,
$$ \left|\frac{g'}{\phi'} \right|\equiv |\phi''(0)|. $$
Applying the Maximum Principle to the holomorphic function $g'/\phi'$, we infer 
$  g' = e^{i\theta}|\phi''(0)| \phi' $ 
for some $\theta\in [0, 2\pi)$. Plugging $g  = (\phi')^{-1} $ in, one further sees that 
$$   -\frac{\phi''}{(\phi')^2} = e^{i\theta}|\phi''(0)| \phi'.$$
Evaluating the above at $0$, we get $e^{i\theta}|\phi''(0)| = -\phi''(0) $. Thus
$   \phi'' = \phi''(0)(\phi')^3.$ 
This is equivalent to 
$ \left((\phi')^{-2}\right)' = -2\phi''(0),  $ 
and so $$\phi' =\frac{1}{ \sqrt{1-2\phi''(0)z}}.$$
Consequently, either $\phi''(0)=0$ in which case $\phi = z$, or \eqref{5} holds with $\alpha$ there equal to $ \phi''(0)$. 
In the latter case, one further computes that $ \phi'''(0) = 3(\phi''(0))^2$. Plugging this and \eqref{5} back to  \eqref{4}, after simplification we immediately observe that the equation is satisfied everywhere on $D_R$ with $R= \frac{1}{2|\alpha|}$. The proof is complete.
\end{proof}

As in the case for  $\phi = z$, the following theorem shows there are many nontrivial solutions with $\phi$ taking the form \eqref{5}. 

\begin{theorem}\label{gens}
Let $\Omega$ be a bounded domain in $\mathbb C^2$. There exists infinitely many smooth solutions to \eqref{mae} of the form \eqref{6} with $\phi$ defined in  \eqref{5} on $\Omega$.
\end{theorem}

\begin{proof}
By the rescaling method mentioned in Section 2,  it suffices to construct a solution of the form   \eqref{6} in a small neighborhood of $0$. 

   In the case when \eqref{5} is taken, for instance with  $\alpha = 1$, in order to obtain a solution to \eqref{mae} of the form \eqref{6} near $0$, according to \eqref{aa}, \eqref{1}, \eqref{bb} and \eqref{cc}, one    looks for   the coefficients $a, b, c$ and $ d$ to satisfy 
\begin{equation*}\label{ma3}
   \left\{  \begin{aligned}
       &a a_{\bar w w} = |a_w|^2 + 4|b_{\bar w}|^2 +1;\\
       &ab_{\bar w w} = 2a_wb_{\bar w};\\
       & ac_{\bar w w} =  a_{   w} c_{\bar w} + 2 b_{\bar w}\bar c_{ w} - 1; \\
       & ad_{\bar w w} =|c_{\bar w}|^2+1.
    \end{aligned}\right.
\end{equation*}
As before letting $\tilde a = \ln a$, then 
\begin{equation}\label{ma4}
   \left\{  \begin{aligned}
      &\tilde a_{\bar w w} =   4e^{-2\tilde a} |b_{\bar w}|^2 + e^{-2\tilde a};\\
       &b_{\bar w w} = 2\tilde a_wb_{\bar w};\\
       &c_{\bar w w} =  \tilde a_{   w} c_{\bar w} + 2 e^{-\tilde a} b_{\bar w}\bar c_{ w} -e^{-\tilde a}; \\
       &d_{\bar w w} =e^{-\tilde a}\left(|c_{\bar w}|^2+1\right).
    \end{aligned}\right.
\end{equation}
Making use of Theorem \ref{pz1}, one can obtain infinitely many solutions to \eqref{ma4} near a neighborhood of $0$. These solutions yield infinitely many solutions to \eqref{mae} of the form \eqref{6} in a small neighborhood of $0$. 
\end{proof}

\section{Donaldson's equation}
 
In this section, we investigate solutions to Donaldson's equation \eqref{de}.  Donaldson's operator  is strictly elliptic when $u_{tt}>0, \Delta u>0$ and $   u_{tt}\Delta u -|\nabla u_t|^2>0 $.   He  constructed in \cite{He} infinitely many entire solutions to    \eqref{de}
on $ \mathbb R\times \mathbb R^m$ of the form
\begin{equation}\label{qde1}
    a_0t^2 + b(x) t + c(x), 
\end{equation} 
where $a_0$ is a positive constant, and $b$ and $ c$ are smooth real functions  on $x\in \mathbb R^m$. He  showed that every solution   of the form \eqref{qde1} must satisfy
    $$\Delta b = 0,\ \ \ \Delta c = \frac{1}{2a_0}(|\nabla b|^2+1). $$
 Note that, according to a result of Warren \cite{Wa}, every convex entire solution to \eqref{de} must be quadratic.
 
We shall   generalize He's idea by  allowing $a_0$ in \eqref{qde1} to  depend  on $x\in \mathbb R^m$. That is, we consider   smooth solutions  of the form  
\begin{equation}\label{qde}
  u(t, x) =   a(x) t^2 + b(x) t + c(x), 
\end{equation}
where $a, b$ and $ c$ are   smooth  real functions  on $x\in \mathbb R^m$ with $a>0$.  
   Plugging it  to \eqref{de}, one gets
$$ 1=  u_{tt}\Delta u -|\nabla u_t|^2  = 2a(t^2\Delta a +t\Delta b+ \Delta c) -|2t\nabla a+\nabla b|^2.$$
Identifying coefficients of $1, t$ and $t^2$, we obtain the following nonlinear differential system
\begin{equation*}
\left\{     \begin{aligned}
        &a\Delta a = 2|\nabla a|^2;\\
        &a\Delta b = 2\nabla a\cdot \nabla b;\\
        &2a\Delta c =|\nabla b|^2+1.
    \end{aligned}\right.
\end{equation*}

Since $a>0$, making use of the transformation  $$\tilde a  = \frac{1}{a},$$ one can  immediately verifies that $\tilde a$ satisfies
\begin{equation}\label{ha}
    \Delta \tilde a =0.
\end{equation} 
There are infinitely many positive harmonic solutions $\tilde a$ on $ B_R$. For each such  $\tilde a$, plugging  $a =\frac{1}{\tilde a}$ into the linear elliptic equation \begin{equation}\label{qde2}
    \Delta b = \frac{2}{a}\nabla a\cdot \nabla b \ \ \text{on}\ \ B_R
\end{equation}    to solve for a smooth solution $b$ on $  B_R$, and then to solve
\begin{equation}\label{qde3}
     \Delta c =\frac{1}{2a}(|\nabla b|^2+1) \ \ \text{on}\ \ B_R
\end{equation}
for a smooth $c$ on $  B_R$.

\begin{theorem} \label{thd}
    For each $R>0$, and a harmonic function $\tilde a$ on $B_R$, define $a: = \tilde a^{-1}$ and let $b$ and $c$ be solutions to  \eqref{qde2}-\eqref{qde3}. Then the function $u$ defined in    \eqref{qde} with these coefficients $a, b$ and $ c$ is a smooth solution to   \eqref{de} on $ \mathbb R \times B_R$.     
\end{theorem}

In the case of entire solutions, Liouville's theorem implies that every positive harmonic function on $\mathbb R^m$ must be a constant. Hence from \eqref{ha} one has $\tilde a\equiv \text{const}$, and further $a\equiv \text{const}$. Moreover,     \eqref{qde2}-\eqref{qde3} gives

\begin{theorem}\label{thde}
    Every entire solution to \eqref{de} on  $\mathbb R\times \mathbb R^m$ of the form \eqref{qde} must satisfy
    $$a\equiv \text{const}, \ \ \Delta b = 0,\ \ \ \Delta c = \frac{1}{2a}(|\nabla b|^2+1)\ \ \text{on}\ \ \mathbb R^m. $$
\end{theorem}

In particular, Theorem \ref{thde} states that in the case of entire solutions,  \eqref{qde}   reduces to the situation \eqref{qde1} in \cite{He}. On the other hand, in the case when $m=1$, since the only harmonic functions are linear functions, $b$ must be linear, and thus $c$ must be quadratic in Theorem \ref{thde}. One immediately obtains the following result.

\begin{cor}\label{co1}
If $m=1$, then every entire solution to \eqref{de} of the form \eqref{qde} must be a quadratic function in $(t, x)\in \mathbb R\times \mathbb R.$
\end{cor}

It is worth noting that Corollary \ref{co} is, in fact,  a special case of \cite{Wa}. This is because when $m=1$, all  solutions obtained in Theorem \ref{thd} and Theorem \ref{thde} are automatically convex, due to  the conditions  $u_{tt}>0$  and $\det (D^2 u)=1>0$. 

\appendix

\section{Resolution of isolated singularity} 
Let $\Omega$ be a domain in $\mathbb R^m, m\ge 2$. A function  $u$ is said to be  a weak solution to a nonlinear differential equation 
$$  \Delta u = f(\cdot, u)\ \ \text{on}\ \ \Omega, $$
if $u\in L^1_{loc}(\Omega), f(\cdot, u)\in L^1_{loc}(\Omega)$ and for any testing function $\phi\in C_c^\infty(\Omega)$, one has
$$ \int_\Omega u\Delta \phi = \int_\Omega f(\cdot,u)\phi. $$
In this Appendix, we shall prove a removable singularity theorem as follows.

\begin{theorem}\label{rem}
   Let $\Omega$ be a domain in $\mathbb R^m$ containing the origin, $m\ge 2$. Let  $u\in C(\Omega)$ if $m=2$, or $u\in L_{loc}^\frac{m}{m-2}(\Omega) $ if $m\ge 3$, and $u$ be a weak solution to 
    $$ \Delta u = f(\cdot,u)\ \ \text{on}\ \ \Omega\setminus \{0\},$$
    where $f\in C^\infty(\mathbb R^m\times \mathbb R)$ with $f\ge 0$ and $ \frac{\partial f}{\partial u}\ge 0$. Then $u\in C^\infty(\Omega)$ and solves 
    $$ \Delta u = f(\cdot,u)\ \ \text{on}\ \ \Omega.$$
\end{theorem}

To prove the theorem, the following Harvey-Polking lemmas (see \cite{HP}) are needed for resolving the isolated singularities. Recall that $D_r$ is the disc in $\mathbb R^2$ of radius $r$. 

\begin{lem}\label{HP}
   If $f\in L^1(D_1)$, and  $u\in C(D_1)$ is a weak solution to $ \Delta u = f$ on $D_1\setminus \{0\}$, then $u $ is a weak solution to
$\Delta u = f$ on $D_1$.
\end{lem}

\begin{proof}
Since $u-u(0)\in C(D_1)$, and  is also a weak solution to $ \Delta u = f$ on $D_1\setminus \{0\}$, without loss of generality assume $u(0)=0$. Given $0<r<1$,    let $\phi^r$ be a smooth function on $D_1$ such that $\phi^r = 1$ on $D_{\frac{r}{2}}$,  $\phi^r = 0$ outside $D_{r}$ and $|\Delta \phi^r|\lesssim r^{-2}$ on $D_r$. Then for any testing function $\phi$ on $D_1$, $(1-\phi^r)\phi$ is a testing function on $D_1 \setminus \{0\} $. Thus 
    $$ \langle  \Delta u -f , (1-\phi^r)\phi\rangle =0, $$
 and so
    \begin{equation*}
        \begin{split}
            \langle \Delta u -f , \phi\rangle =  \langle  \Delta u -f , \phi^r\phi\rangle  
            =  \langle  u,   \Delta(\phi^r\phi)\rangle -  \langle  f , \phi^r\phi\rangle.
        \end{split}
    \end{equation*}
    Passing $r$ to $0$, since $f\in L^1(D_1)$, $$ \langle  f , \phi^r\phi\rangle \lesssim \int_{D_{r} }|f| \rightarrow 0.$$
  On the other hand,   $$\langle  u,  \Delta(\phi^r\phi)\rangle \lesssim r^{-2} \int_{D_r }|u|\le \max_{D_r}|u|  
    \rightarrow 0.$$   
  We thus have the desired identity $  \langle  \Delta u -f , \phi\rangle =0.  $   
\end{proof}

The continuity assumption on $u$ in Lemma \ref{HP} can not be dropped, as demonstrated by the following example.  

\begin{example}
 Let $u$ be a smooth solution to 
 $$ \Delta u = 4e^{2u} \ \ \text{on}\ \ D_1. $$
For instance, one can check that $u = - \ln (1-|x|^2)$ is such a solution. 
Consequently, $v: = u -\ln |x|  $ is a  smooth solution to 
 $$ \Delta v =4|x|^2 e^{2v} \ \ \text{on}\ \ D_1\setminus \{0\}. $$
However, $v$ is not a weak solution to 
    $$\Delta v =4|x|^2e^{2v} \ \ \text{on}\ \ D_1,$$
    since  $-\frac{1}{2\pi}\ln|x|$ is the fundamental solution to $\Delta$. Note that $v\notin C(D_1)$, so Lemma \ref{HP} does not apply.
\end{example}

 By slightly adjusting the proof of Lemma \ref{HP}, one can   resolve  isolated singularities in higher dimensional case, under a weaker assumption on the regularity of $u$ than continuity. 

\begin{lem}\label{HPn}
  Let $B_1$ be the unit ball in $\mathbb R^m, m\ge 3$.   If $f\in L_{loc}^1(B_1)$, and  $u\in L_{loc}^\frac{m}{m-2}(B_1)$ is a weak solution to $ \Delta u = f$ on $B_1\setminus \{0\}$, then $u $ is a weak solution to
$\Delta u = f$ on $B_1$.
\end{lem}

\begin{proof}
    In view of the proof to Lemma \ref{HP}, we only need to verify that $ r^{-2} \int_{B_r }|u|\rightarrow 0 $. This is obvious by H\"older's inequality:
    $$ r^{-2} \int_{B_r }|u|\le r^{-2} \left(\int_{B_r}|u|^\frac{m}{m-2}\right)^{\frac{m-2}{m}}\left(\int_{B_r} 1\right)^\frac{2}{m} \lesssim \left(\int_{B_r}|u|^\frac{m}{m-2}\right)^{\frac{m-2}{m}}\rightarrow 0$$
    as $r\rightarrow 0$.
\end{proof}

Next, we prove that under suitable smoothness and growth assumptions on $f$, any weak solution in fact belongs to a higher regularity class, thereby becoming a classical solution. This together with Lemma \ref{HP} and Lemma \ref{HPn} proves Theorem \ref{rem}. 

\begin{pro}\label{reg}
   Let $\Omega$ be a domain in $\mathbb R^m$.  Let $u$ be a  weak solution to
    $$  \Delta u = f(\cdot, u)\ \ \text{on}\ \ \Omega,$$
    where $f\in C^\infty(\mathbb R^m\times \mathbb R)$ with $f\ge 0, \frac{\partial f}{\partial u}\ge 0$. Then $u$   is smooth on $\Omega$. 
\end{pro}

\begin{proof}
  Since $\Delta u\ge 0$ in the sense of distributions, $u$ become subharmonic on $\Omega$ after redefining its values on a measure zero set.   See, for instance, \cite[Theorem 3.2.11]{Ho}. Hence $u$ is upper semi-continuous on $\Omega$. In particular, for every $V\subset\subset \Omega$, there exists a constant $c>0$ such that  $u\le c$ on $V$. Consequently, by the monotonicity of $f$ with respect to $u$, 
  $$0\le f(\cdot, u)\le  f(\cdot,  c)  \ \text{on}\ \ V. $$
This implies $\Delta u\in L^\infty(V)$. 

From the  $W^{2, p}$ theory of elliptic equations, we deduce that $u\in W^{2, p}(V)$ for any $p<\infty$. The Sobolev embedding theorem then yields $u\in C^{1,\alpha}(V)$ for all $0<\alpha<1$. Applying Schauder theory, we can further obtain $u\in C^{3, \alpha}(V)$. A standard bootstrapping argument eventually gives 
 $u\in C^\infty(V)$.
\end{proof}

An immediate consequence of (the proof to) Theorem \ref{rem}  is an improved regularity   for every weak solution to  $ \Delta u = e^{cu}, $ where $c$ is  a positive constant.     

\begin{cor}
     Let $\Omega$ be a domain in $\mathbb R^m$ and $c$ be a positive constant. Then every weak solution  to 
    $$ \Delta u = e^{cu}\ \ \text{on}\ \ \Omega$$
    must be smooth on $\Omega$. 
\end{cor}

It is natural to ask whether a similar regularity-improving property as in  Theorem \ref{rem} still holds if the condition $\frac{\partial f}{\partial u}\ge 0 $ is dropped, as our approach does not extend to this case. For instance,   the equation  $\Delta u =  e^{-u} $, where $f(u) := e^{-u}$ satisfies $\frac{\partial f}{\partial u}< 0 $. The following example demonstrates that  the property fails if $m\ge 3 $.    The situation for $m=2$ remains unclear.  

\begin{example}
Let $m\ge 3$. A direct computation can verify that $u(x)= 2\ln |x|-\ln(2m-4)$ is a smooth solution to 
    $$ \Delta u = e^{-u}\ \ \text{on}\ \ \mathbb R^m\setminus\{0\}.$$
On the other hand, $e^{-u} = \frac{2m-4}{|x|^2}\in L^1_{loc}(\mathbb R^m)$, and $u\in L^p_{loc}(\mathbb R^m)$ for all $p<\infty$. According to Lemma \ref{HPn}, $u$ is a weak solution to 
$$\Delta u = e^{-u}\ \ \text{on}\ \ \mathbb R^m. $$
However, $u$ is not even continuous at $0$.
\end{example}

\bibliographystyle{alphaspecial}

\fontsize{11}{11}\selectfont

\vspace{0.7cm}

\noindent pan@pfw.edu,

\vspace{0.2 cm}

\noindent Department of Mathematical Sciences, Purdue University Fort Wayne, Fort Wayne, IN 46805-1499, USA.\\

\noindent zhangyu@pfw.edu,

\vspace{0.2 cm}

\noindent Department of Mathematical Sciences, Purdue University Fort Wayne, Fort Wayne, IN 46805-1499, USA.\\
\end{document}